\documentclass{amsart}

\usepackage{amssymb,dirtytalk,comment,tikz-cd,tikz,color}
\usepackage[hidelinks]{hyperref}

\newtheorem{theorem}{Theorem}[section]
\newtheorem{lemma}[theorem]{Lemma}

\theoremstyle{definition}
\newtheorem{definition}[theorem]{Definition}
\newtheorem{example}[theorem]{Example}

\newtheorem{problem}[theorem]{Problem}

\theoremstyle{remark}
\newtheorem{remark}[theorem]{Remark}

\numberwithin{equation}{section}

\DeclareMathOperator{\Aut}{Aut}
\DeclareMathOperator{\End}{End}

\DeclareMathOperator{\Perm}{Perm}

\newcommand{\N}{\mathbb{N}}
\newcommand{\Z}{\mathbb{Z}}

\begin{document}

\title[Types for which every Hopf--Galois correspondence is bijective]{Classification of the types for which every Hopf--Galois correspondence is bijective}

\author{Lorenzo Stefanello}
\address{Dipartimento di Matematica,
          Universit\`a di Pisa,
          Largo Bruno Pontecorvo 5, 56127 Pisa, Italy}
\email{lorenzo.stefanello@phd.unipi.it}
\urladdr{https://people.dm.unipi.it/stefanello/}
\thanks{The first-named author was a member of GNSAGA - INdAM}

\author{Cindy (Sin Yi) Tsang}
\address{Department of Mathematics, Ochanomizu University, 2-1-1 Otsuka, Bunkyo-ku, Tokyo, Japan}
\email{tsang.sin.yi@ocha.ac.jp}
\urladdr{http://sites.google.com/site/cindysinyitsang/}



\begin{abstract}
Let $L/K$ be any finite Galois extension with Galois group $G$. It is known by Chase and Sweedler that the Hopf--Galois correspondence is injective for every Hopf--Galois structure on $L/K$, but it need not be bijective in general. Hopf--Galois structures are known to be related to skew braces, and recently, the first-named author and Trappeniers proposed a new version of this connection with the property that the intermediate fields of $L/K$ in the image of the Hopf--Galois correspondence are in bijection with the left ideals of the associated skew brace. As an application, they classified the groups $G$ for which the Hopf--Galois correspondence is bijective for every Hopf--Galois structure on any $G$-Galois extension. In this paper, using a similar approach, we shall classify the groups $N$ for which the Hopf--Galois correspondence is bijective for every Hopf--Galois structure of type $N$ on any Galois extension.
\end{abstract}

\maketitle

\section{Introduction}\label{sec:intro}

Let $L/K$ be any finite extension of fields. A \emph{Hopf--Galois structure} on $L/K$ is a $K$-Hopf algebra $H$ together with an action $\star$ on $L$ such that $L$ is an $H$-module algebra and the natural $K$-linear map
\[ L\otimes_K H \longrightarrow \End_K(L),\quad \ell\otimes h \mapsto (x\mapsto \ell(h\star x))\]
is bijective. In this case, for each $K$-Hopf subalgebra $H'$ of $H$, we attach to it the intermediate field of $L/K$ given by
\[ L^{H'} = \{x\in L \colon h'\star x = \varepsilon(h')x\mbox{ for all }h'\in H'\},\]
where $\varepsilon \colon H \longrightarrow K$ denotes the counit of $H$. This yields the so-called \emph{Hopf--Galois correspondence} for $H$ from the $K$-Hopf subalgebras of $H$ to the intermediate fields of $L/K$. It is well-known that this correspondence is always injective \cite{CS}.

\vspace{2mm}

The notion of Hopf--Galois structure was first introduced by Chase and Sweedler \cite{CS}. The original motivation was to study purely inseparable extensions, but nowadays Hopf–Galois theory is mostly considered for separable extensions. There is a particular interest in the case of finite Galois extensions because it has applications to the study of Galois module structure of rings of integers; see \cite{Childs book}. For finite separable extensions, it is known by the groundbreaking work of Greither and Pareigis \cite{GP} that Hopf--Galois structures correspond bijectively to certain regular permutation groups. Their result is the foundation of Hopf--Galois theory for separable extensions---it reduces the study of Hopf--Galois structures to a completely group-theoretic problem. Let us briefly recall this correspondence in the special case of Galois extensions.

\vspace{2mm}

Let $L/K$ be any finite Galois extension of fields with Galois group $G$. Also let $\Perm(G)$ denote the symmetric group of $G$, and let
\[ \begin{cases}
\lambda \colon G\longrightarrow \Perm(G), &\lambda(\sigma) = (x\mapsto \sigma x)\\
\rho \colon G\longrightarrow \Perm(G), &\rho(\sigma) = (x\mapsto x\sigma^{-1})
\end{cases}\]
be the left and right regular representations of $G$. For any regular subgroup $N$ of $\Perm(G)$ that is normalized by $\lambda(G)$, we associate to it the Hopf--Galois structure 
\[ L[N]^{G} = \left\{\sum_{\eta\in N}\ell_\eta\eta \in L[N] \colon  \sigma(\ell_\eta) = \ell_{\lambda(\sigma)\eta\lambda(\sigma)^{-1}} \mbox{ for all }\sigma\in G,\ \eta\in N\right\}\]
whose action on $L$ is defined by
\begin{equation}\label{GP action} \left(\sum_{\eta\in N}\ell_\eta\eta\right)\star x = \sum_{\eta\in N} \ell_\eta \eta^{-1}(1_G)(x)\quad\mbox{for all }x\in L.
\end{equation}
This procedure gives a bijection between such regular subgroups and Hopf--Galois structures on $L/K$ (up to a natural notion of isomorphism). The \emph{type} of a Hopf--Galois structure $H$ on $L/K$ is the isomorphism class of the corresponding regular subgroup, or equivalently, the isomorphism class of a finite group $N$ for which 
\[ L\otimes_K H \simeq L[N]\]
as $L$-Hopf algebras; such a group $N$ is unique up to isomorphism. 

%Up to isomorphism, every Hopf--Galois structure on $L/K$ arises in this way \cite{GP}. The \emph{type} of any Hopf--Galois structure $H$ on $L/K$ is defined to be the isomorphism class of the corresponding regular subgroup $N$, or equivalently, the isomorphism class of any finite group $N$ for which $L\otimes_K H \simeq L[N]$ as $L$-Hopf algebras. 

\vspace{2mm}

There are two obvious Hopf--Galois structures of type $G$ on $L/K$---the so-called \emph{classical} structure $H_{\rho}\simeq K[G]$ corresponding to $\rho(G)$, and the \emph{canonical nonclassical} structure $H_\lambda$ corresponding to $\lambda(G)$; their actions on $L$ are as defined in (\ref{GP action}). Note that $H_\rho$ and $H_\lambda$ coincide, namely $\rho(G) = \lambda(G)$, if and only if $G$ is abelian. The Hopf--Galois correspondence for $H_\rho$ is the usual Galois correspondence from Galois theory and so is always bijective. However, the Hopf--Galois correspondence for $H_\lambda$ is bijective (recall that injectivity always holds by \cite{CS}) if and only if every subgroup of $G$ is normal, because its image consists precisely of the normal intermediate fields of $L/K$ as is known by \cite{GP}.

\vspace{2mm}

The (non)bijectivity of the Hopf--Galois correspondence is a natural problem of interest; see \cite{GC1,GC2,GC3,EG} for some related results. More generally, knowing exactly which of the intermediate fields lie in the image can be fruitful in Galois module theory. For example, a key ingredient of \cite{Byott} (where the Galois module structure of valuation rings was studied for Galois extensions of $p$-adic fields of degree $p^2$) was the existence of a suitable intermediate field in the image of the Hopf--Galois correspondence. Instead of considering a specific Hopf--Galois correspondence, the first-named author and Trappeniers \cite{ST} fixed the Galois group $G$ and asked for conditions under which the Hopf--Galois correspondence is bijective for every Hopf--Galois structure on a $G$-Galois extension. Clearly any such group $G$ is necessarily Hamiltonian or abelian because of the canonical nonclassical structure $H_\lambda$, and a complete classification was obtained in \cite[Theorem 4.24]{ST}, as follows.

\begin{theorem}\label{thm:G} For any finite group $G$, the following are equivalent:
\begin{enumerate}
\item The Hopf--Galois correspondence is bijective for every Hopf--Galois structure on any finite $G$-Galois extension.
\item The group $G$ is cyclic and $q\nmid p-1$ for all prime divisors $p,q$ of $|G|$.
\end{enumerate}
\end{theorem} 

In this paper, we shall fix the type $N$ of the Hopf--Galois structures instead and look for conditions under which the Hopf--Galois correspondence is bijective for every Hopf--Galois structure of type $N$ on any Galois extension. Clearly such a group $N$ is necessarily Hamiltonian or abelian, as one can see by considering the canonical nonclassical structure $H_\lambda$  on an $N$-Galois extension. We are able to obtain a complete classification and below is our main result. Note that the criteria on $N$ are similar to but slightly different from those on $G$ in Theorem \ref{thm:G}.

\begin{theorem}\label{thm:N} For any finite group $N$, the following are equivalent:
\begin{enumerate}
\item The Hopf--Galois correspondence is bijective for every Hopf--Galois structure of type $N$ on any finite Galois extension.
\item The group $N$ is isomorphic to $C_2$ or $C_2\times C_2$, or is cyclic of odd order and $q\nmid p-1$ for all prime divisors $p,q$ of $|N|$.
\end{enumerate}
\end{theorem}

To prove Theorem \ref{thm:N}, we shall follow the strategy of \cite{ST} and use the connection between Hopf--Galois structures and skew braces (to be recalled in Section \ref{sec:skew braces}).

\section{Connection with skew braces}\label{sec:skew braces}

Skew braces are an algebraic structure introduced by Guarnieri and Vendramin \cite{GV} as a tool to study the non-degenerate set-theoretic solutions to the Yang--Baxter equation. They are a generalisation of braces as introduced by Rump \cite{Rump0}.

\begin{definition} A \emph{skew (left) brace} is a set $B = (B,\cdot,\circ)$ equipped with two group operations $\cdot$ and $\circ$ such that the so-called brace relation
\[ a\circ (b\cdot c) = (a\circ b)\cdot a^{-1}\cdot (a\circ c)\]
holds for all $a,b,c\in B$. For $a\in B$, we write $a^{-1}$ for the inverse of $a$ in $(B,\cdot)$ and $\overline{a}$ for the inverse of $a$ in $(B,\circ)$. One calls $(B,\cdot)$ and $(B,\circ)$, respectively, the \emph{additive group} and \emph{multiplicative group} of the skew brace $B$. They share the same identity element, which we denote by $1$. A \emph{brace} is a skew brace with abelian additive group.
\end{definition}

\begin{definition}    
Given any skew brace $B = (B,\cdot,\circ)$, for each $a\in B$, it is easy to verify, using the brace relation, that the map
\[ \gamma_a \colon B \longrightarrow B,\quad \gamma_a(b) = a^{-1}\cdot (a\circ b)\]
is an automorphism on $(B,\cdot)$. It is also well-known \cite{GV} that
\[ \gamma \colon (B,\circ) \longrightarrow \Aut(B,\cdot),\quad \gamma(a) = \gamma_a\]
is a group homomorphism, which we shall refer to as the \emph{gamma function} of $B$.
\end{definition}

There are two obvious ways to construct a skew brace from a group $(B,\cdot)$---the \emph{trivial} skew brace $(B,\cdot,\cdot)$, and the \emph{almost trivial} skew brace $(B,\cdot,\cdot^{\mbox{\tiny op}})$, where $\cdot^{\mbox{\tiny op}}$ denotes the opposite operation of $\cdot$ that is defined by $a\cdot^{\mbox{\tiny op}}b = b\cdot a$ for all $a,b\in B$. Note that they give the same skew brace if and only if $(B,\cdot)$ is abelian.

\vspace{2mm}

It is known that both Hopf--Galois structures \cite[Chapter 2]{Childs book} and skew braces \cite[Section 4]{GV} are related to regular subgroups of the holomorph of a group, so there is a connection between Hopf--Galois structures and skew braces \cite[Appendix]{SV}. Childs \cite{GC2} used this connection to translate the Hopf--Galois correspondence problem to the study of certain substructures of the associated skew brace. But the substructures that appeared are not natural objects that one usually finds in the theory of skew braces. Motivated by this, the first-named author and Trappeniers \cite{ST} proposed a new version of the connection; an important idea was to use opposite skew braces as defined by Koch and Truman \cite{opposite}. This new version of the connection \cite[Theorem 3.1]{ST}, to be recalled below, is explicit and is bijective.
 
\begin{theorem}\label{thm:corr1} Let $L/K$ be a finite Galois extension of fields with Galois group $(G,\circ)$. Then there exists a bijective correspondence between
\begin{enumerate}
\item the operations $\cdot$ such that $(G,\cdot,\circ)$ is a skew brace, and
\item the Hopf--Galois structures on $L/K$ (up to isomorphism).
\end{enumerate}
Specifically, such an operation $\cdot$ is associated to the Hopf--Galois structure
\[ L[G_\cdot]^{G_\circ} = \left\{\sum_{\tau\in G}\ell_\tau\tau \in L[G_\cdot] \colon \sigma(\ell_\tau) = \ell_{\gamma_{\sigma}(\tau)}\mbox{ for all }\sigma,\tau\in G\right\}\]
whose action on $L$ is defined by
\begin{equation}\label{star action} \left(\sum_{\tau\in G}\ell_\tau\tau\right)\star x = \sum_{\tau\in G} \ell_\tau \tau(x)\quad\mbox{for all }x\in L.
\end{equation}
Here we are writing $G_\cdot = (G,\cdot)$ and $G_\circ = (G,\circ)$ for simplicity, and $\gamma$ denotes the gamma function of $(G,\cdot,\circ)$.
\end{theorem}

\begin{example} For the operation $\cdot$ in Theorem \ref{thm:corr1}, the choice $\circ$ (corresponding to the trivial skew brace) is associated to the classical structure $H_\rho$ on $L/K$. Similarly, the choice $\circ^{\mbox{\tiny op}}$ (corresponding to the almost trivial skew brace) is associated to the canonical nonclassical structure $H_\lambda$ on $L/K$. 
\end{example}

Under this new version of the connection, the Hopf--Galois correspondence problem may now be rephrased in terms of left ideals of the associated skew brace. Left ideals are a natural algebraic substructure that appears in the study of skew braces.

\begin{definition} Let $B = (B,\cdot,\circ)$ be a skew brace with gamma function $\gamma$. A \emph{left ideal} of $B$ is a subgroup $I$ of $(B,\cdot)$ for which $\gamma_a(I) \subseteq I$ for all $a\in B$. In this case, one easily checks that $I$ is automatically a subgroup of $(B,\circ)$.
\end{definition}

The following result is from \cite[Corollary 4.1]{ST}.

\begin{theorem}\label{thm:corr2} Let $L/K$ be a finite Galois extension of fields with Galois group $(G,\circ)$. Let $\cdot$ be any operation such that $(G,\cdot,\circ)$ is a skew brace, and let $L[G_\cdot]^{G_\circ}$ be the associated Hopf--Galois structure. For any subgroup $G'$ of $(G,\circ)$, the following are equivalent:
\begin{enumerate}
\item $G'$ is a left ideal of $(G,\cdot,\circ)$;
\item $L^{G'}$ lies in the image of the Hopf--Galois correspondence for $L[G_\cdot]^{G_\circ}$.
\end{enumerate}
%Then there exists a bijective correspondence between
%\begin{enumerate}
%\item the left ideals of $(G,\cdot,\circ)$, and
%\item the intermediate fields of $L/K$ that lie in the image of the Hopf--Galois correspondence for $L[G_\cdot]^{G_\circ}$.
%\end{enumerate}
In particular, the Hopf--Galois correspondence for $L[G_\cdot]^{G_\circ}$ is bijective if and only if every subgroup of $(G,\circ)$ is a left ideal of $(G,\cdot,\circ)$.
\end{theorem}

\begin{remark}\label{ex:order}In the setting of Theorem \ref{thm:corr2}, if the number of subgroups of $(G,\circ)$ of a given order is larger than that of $(G,\cdot)$, then the subgroups of $(G,\circ)$ cannot all be left ideals of $(G,\cdot,\circ)$.  Hence, the Hopf--Galois correspondence for $L[G_\cdot]^{G_\circ}$ is not bijective in this case.
\end{remark}

\begin{remark}\label{ex:subgp}In the setting of Theorem \ref{thm:corr2}, if the number of subgroups of $(G,\cdot)$ equals that of $(G,\circ)$, but there is some subgroup of $(G,\cdot)$ that is not a subgroup of $(G,\circ)$, then the subgroups of $(G,\circ)$ cannot all be left ideals of $(G,\cdot,\circ)$. Hence, the Hopf--Galois correspondence for $L[G_\cdot]^{G_\circ}$ is not bijective in this case.
\end{remark}

\begin{remark}\label{ex:char subgp}In the setting of Theorem \ref{thm:corr2}, if the number of characteristic subgroups of $(G,\cdot)$ equals that of subgroups of $(G,\circ)$, then the subgroups of $(G,\circ)$ are all left ideals of $(G,\cdot,\circ)$, as noted in \cite[Proposition 4.19]{ST}. This follows from the obvious fact that every characteristic subgroup of $(G,\cdot)$ is a left ideal of $(G,\cdot,\circ)$. Hence, the Hopf--Galois correspondence for $L[G_\cdot]^{G_\circ}$ is bijective in this case. 
\end{remark}

By Theorem \ref{thm:corr2} and the fact that every finite group arises as the Galois group of some finite Galois extension of fields, the proof of Theorem \ref{thm:N} may be reduced to the following skew-brace-theoretic problem.

\begin{problem}\label{problem} Classify the finite groups $N = (N,\cdot)$ satisfying the following condition: Every subgroup of $(N,\circ)$ is a left ideal of $(N,\cdot,\circ)$ for all operations $\circ$ such that $(N,\cdot,\circ)$ is a skew brace. \end{problem}

For simplicity, let us say that a finite group $N = (N,\cdot)$ is \emph{good} if it satisfies the condition stated in Problem \ref{problem}, and \emph{bad} otherwise.

\vspace{2mm}

We end this section with a useful lemma (cf. \cite[Lemma 4.23]{ST}). 

\begin{lemma}\label{lem:direct prod} Let $N = (N,\cdot)$ be a finite group. If $N = M\times M'$ is a direct product of two subgroups $M,M'$ for which $M$ is bad, then $N$ is also bad. 
\end{lemma}
\begin{proof} The hypothesis means that there exists an operation $\circ$ for which $(M,\cdot,\circ)$ is a skew brace and some subgroup $S$ of $(M,\circ)$ is not a left ideal of $(M,\cdot,\circ)$. Then clearly $(N,\cdot,\circ)$ is also a skew brace, where we define
\[ (m_1,m_1')\circ (m_2,m_2') = (m_1\circ m_2,m_1'\cdot m_2')\]
for all $m_1,m_2\in M$ and $m_1',m_2'\in M'$. It is also obvious that $S\times \{1\}$ is a subgroup of $(N,\circ)$ that is not a left ideal of $(N,\cdot,\circ)$. Thus, indeed $N$ is bad.
\end{proof}

\section{Examples of bad groups}\label{sec:ex}

In this section, we shall give some examples of bad groups and explicitly construct, via skew braces, Hopf--Galois structures (the action is understood to be the one given by (\ref{star action})) for which the Hopf--Galois correspondence is not bijective. We shall need them for the proof of (1) $\Rightarrow$ (2) of Theorem \ref{thm:N}.

\vspace{2mm}

The first two examples are $Q_8$ and $C_2^3$. We remark that skew braces of order $8$ (and more generally order $p^3$ for a prime $p$) were enumerated by Nejabati Zenouz \cite{NZ}, and those that are braces were previously classified by Bachiller \cite{p3}.

\begin{example}\label{ex:Q8} Consider the quaternion group $Q_8=(Q_8,\cdot)$ with presentation
\[ Q_8 = \langle \sigma, \tau \mid \sigma^4=1, \ \sigma^2=\tau^2,\ \tau\sigma\tau^{-1}=\sigma^{-1}\rangle.\]
Let $\psi \colon Q_8\longrightarrow Q_8$ denote the automorphism defined by
\[ \psi(\sigma) = \sigma,\quad \psi(\tau) =\sigma\tau. \]
Since $[Q_8,Q_8] = Z(Q_8)$, clearly $\psi[Q_8,Q_8] \subseteq Z(Q_8)$ holds. We can then define
\[ \delta\circ \delta' = \delta\cdot \psi(\delta)\cdot \delta'\cdot \psi(\delta)^{-1}\]
for all $\delta,\delta'\in Q_8$, and this yields a skew brace $(Q_8,\cdot,\circ)$ by \cite[Theorem 1.2]{endo}. It is routine, once we note that
\[ \sigma \tau\sigma^{-1}=\tau^{-1},\quad
(\sigma\tau)\sigma(\sigma\tau)^{-1}=\sigma^{-1},\quad (\sigma\tau)\tau(\sigma\tau)^{-1}=\tau^{-1},\]
to explicitly compute that 
\[ \sigma^{i}\tau^{j}\circ \sigma^{r}\tau^{s}
= \sigma^{i+r}\tau^{j+(-1)^{i+j}s}\]
for all $i,j,r,s\in \Z$. Note that $(Q_8,\circ)\simeq D_8$ is also generated by $\sigma,\tau$ because
\[ \underbrace{\sigma\circ\cdots \circ\sigma}_{\mbox{\tiny $4$ times}} = \sigma^4 =1,\quad \tau\circ\tau = 1,\quad \tau\circ\sigma\circ\tau = \sigma^{-1} = \overline{\sigma}.\]
Since $D_8$ has more subgroups of order $2$ than $Q_8$, not all subgroups of $(Q_8,\circ)$ are left ideals of $(Q_8,\cdot,\circ)$ as noted in Remark \ref{ex:order}. Thus the group $Q_8$ is bad. 

\vspace{2mm}

We can construct the associated Hopf--Galois structure explicitly: observe that
\[ \gamma_\sigma(\sigma^i\tau^j) = \sigma^i\tau^{-j} = \sigma^{i+2}\tau^{2-j},\quad \gamma_\tau(\sigma^i\tau^j) = \sigma^{-i}\tau^{-j}=\sigma^{-i+2}\tau^{2-j}\]
for all $i,j\in \Z$. This means that if $L/K$ is a $D_8$-Galois extension and we identify its Galois group with $(Q_8,\circ)$, then we obtain the Hopf--Galois structure %with Hopf algebra
\[ H = \left\{
\sum_{i\in \mathbb{Z}/4\mathbb{Z}}\sum_{j\in\{0,1\}}\ell_{i,j}\sigma^i\tau^j\in L[Q_8]\colon
\begin{array}{c}
\sigma(\ell_{i,0}) = \ell_{i,0},\, \sigma(\ell_{i,1}) = \ell_{i+2,1},\\
\tau(\ell_{i,0}) = \ell_{-i,0},\, \tau(\ell_{i,1}) = \ell_{-i+2,1}\mbox{ for all }i
\end{array}
\right\}\]
of type $Q_8$ for which the Hopf--Galois correspondence is not bijective.
\end{example}

\begin{example}\label{ex:C2C2C2}
Consider the elementary abelian $2$-group $C_2^3=(C_2^3,\cdot)$ of rank $3$ with presentation
\[ C_2^3 = \langle \sigma,\tau,\upsilon \mid \sigma^2 = \tau^2 = \upsilon^2=1, \ \sigma\tau=\tau\sigma,\ \sigma\upsilon=\upsilon\sigma,\ \tau\upsilon=\upsilon\tau \rangle. \]
We know from \cite[Theorem 3.1]{p3} that by defining
\[ \sigma^{i}\tau^{j}\upsilon^{k} \circ \sigma^{r}\tau^{s}\upsilon^{t} = \sigma^{i+r+jt+ks}\tau^{j+s}\upsilon^{k+t} \]
for all $i,j,k,r,s,t\in \Z$, we get a brace $(C_2^3,\cdot,\circ)$. It is easy to see that $(C_2^3,\circ)\simeq C_2^3$ is also generated by $\sigma,\tau,\upsilon$. Since $\langle \tau,\upsilon\rangle$ is a subgroup of $(C_2^3,\cdot)$ but not of $(C_2^3,\circ)$, not all subgroups of $(C_2^3,\circ)$ are left ideals of $(C_2^3,\cdot,\circ)$ by Remark \ref{ex:subgp}. This shows that the group $C_2^3$ is bad.

\vspace{2mm}

We can construct the associated Hopf--Galois structure explicitly: observe that
\[ \gamma_\sigma(\sigma^i\tau^j\upsilon^k)=\sigma^i\tau^j\upsilon^k,
\quad \gamma_\tau(\sigma^i\tau^j\upsilon^k) = \sigma^{i+k}\tau^j\upsilon^k,
\quad \gamma_\upsilon(\sigma^i\tau^j\upsilon^k) = \sigma^{i+j}\tau^j\upsilon^k\]
for all $i,j,k\in \Z$. This means that if $L/K$ is a $C_2^3$-Galois extension and we identify its Galois group with $(C_2^3,\circ)$, then we obtain the Hopf--Galois structure %with Hopf algebra
\[ H = \left\{
\sum_{i,j,k\in \mathbb{Z}/2\mathbb{Z}}\ell_{i,j,k}\sigma^i\tau^j\upsilon^k\in L[C_2^3]\colon
\begin{array}{cc}
\sigma(\ell_{i,j,k}) = \ell_{i,j,k},\ \tau(\ell_{i,j,k})=\ell_{i+k,j,k},
\\[3pt]
\upsilon(\ell_{i,j,k})=\ell_{i+j,j,k}\mbox{ for all }i,j,k
\end{array}
\right\}\]
of type $C_2^3$ for which the Hopf--Galois correspondence is not bijective.
\end{example}

The next two examples concern cyclic groups. We remark that for finite braces with cyclic additive group, the isomorphism classes of the multiplicative group that can occur were already classified by Rump \cite{Rump}.

\begin{example}\label{ex:2n}Consider the cyclic group $C_n=(C_{n},\cdot)$ of even order $n$ with generator $\sigma$. Since $n$ is even, we can define 
\[ \sigma^i \circ \sigma^j = \sigma^{i+ (-1)^ij}\]
for all $i,j\in \Z$. One easily verifies that $(C_n,\circ)$ is indeed a group and that $(C_n,\cdot,\circ)$ is a brace. Notice that $(C_n,\circ)\simeq D_{n}$ is a dihedral group (with the convention that $D_2=C_2$ and $D_4 = C_2\times C_2)$ generated by $\sigma,\sigma^2$ because
\[ \underbrace{\sigma^2\circ\cdots\circ \sigma^2}_{\mbox{\tiny $n$ times}} = (\sigma^2)^n = 1,\quad
\sigma\circ \sigma =1,\quad
\sigma\circ \sigma^2\circ\sigma = (\sigma^2)^{-1} = \overline{\sigma^2}.\]
Suppose now that $n\geq 4$. Then $D_n$ has more subgroups of order $2$ than $C_n$, so not all subgroups of $(C_n,\circ)$ are left ideals of $(C_n,\cdot,\circ)$ as noted in Remark \ref{ex:order}. It then follows that the group $C_n$ is bad for $n\geq 4$.

\vspace{2mm}

We can construct the associated Hopf--Galois structure explicitly: observe that
\[ \gamma_\sigma(\sigma^i) = \sigma^{-i},\quad
\gamma_{\sigma^2}(\sigma^i) = \sigma^i \]
for all $i\in \Z$. This means that if $L/K$ is a $D_n$-Galois extension and we identify its Galois group with $(C_n,\circ)$, then we obtain the Hopf--Galois structure
\[ H = \left\{
\sum_{i\in\mathbb{Z}/n\mathbb{Z}}\ell_{i}\sigma^i\in L[C_n]\colon
\sigma(\ell_{i}) = \ell_{-i},\ \sigma^2(\ell_{i})=\ell_{i}\mbox{ for all }i
\right\}\]
of type $C_n$ for which the Hopf--Galois correspondence is not bijective. 
\end{example}

\begin{example}\label{ex:pq}Consider the direct product $C_{p^{n}}\times C_{q^m}=(C_{p^{n}}\times C_{q^m},\cdot)$, where $p,q$ are primes for which $q\mid p-1$ and $m,n\geq 1$. Let $C_{p^n} = \langle\sigma\rangle$ and $C_{q^{m}} = \langle\tau\rangle$. Since $q\mid p-1$, there exists $\kappa\in\Z$ such that $\kappa$ mod $p^n$ has multiplicative order $q$. We can then define a natural homomorphism by setting
\[ \varphi\colon C_{q^m} \longrightarrow \Aut(C_{p^n}),\quad \varphi(\tau) = (\sigma\mapsto \sigma^\kappa).\]
It is known, by \cite[Example 1.4]{GV} for example, that by setting
\[ (\delta,\xi) \circ (\delta',\xi') = (\delta\varphi(\xi)(\delta'),\xi\xi') \]
for all $\delta,\delta'\in C_{p^n}$ and $\xi,\xi'\in C_{q^m}$, namely
\[ (C_{p^n} \times C_{q^m},\circ) = C_{p^n}\rtimes_\varphi C_{q^m}\]
is the semidirect product defined by $\varphi$, we obtain a brace $(C_{p^n}\times C_{q^m},\cdot,\circ)$. Clearly $\{1\} \times C_{q^m}$ is a subgroup of $(C_{p^n}\times C_{q^m},\circ)$ but it is not normal because $\varphi$ is non-trivial. Thus $(C_{p^n}\times C_{q^m},\circ)$ has more Sylow $q$-subgroups than $(C_{p^n}\times C_{q^m},\cdot)$, so not all subgroups of $(C_{p^n}\times C_{q^m},\circ)$ are left ideals of $(C_{p^n}\times C_{q^m},\cdot,\circ)$ by Remark \ref{ex:order}. This implies that the group $C_{p^n}\times C_{q^m}$ is bad when $q\mid p-1$.
%Clearly $\{1\} \times C_{q^m}$ is a subgroup of $(C_{p^n}\times C_{q^m},\circ)$, so its conjugate
%\begin{align*} S 
%&= (\sigma,1)\circ (\{1\} \times C_{q^m}) \circ \overline{(\sigma,1)}\\
%& = (\sigma,1)\circ (\{1\} \times C_{q^m}) \circ (\sigma^{-1},1)\\
%&= \{ (\sigma\varphi_\xi(\sigma^{-1}),\delta) : \xi\in C_{q^m}\}
%\end{align*}
%is also a subgroup of $(C_{p^n} \times C_{q^m},\circ)$. But $S$ not a left ideal of $(C_{p^n} \times C_{q^m},\cdot,\circ)$. To see why, note that it suffices to show that  
%\[ \gamma_{(1,\tau)}( \sigma\varphi_\tau(\sigma^{-1}),\tau) = (\varphi_\tau(\sigma\varphi_\tau(\sigma^{-1})),\tau) \]
%is not element of $S$. Suppose for contradiction that it lies in $S$. Then
%\[ \sigma\varphi_\tau(\sigma^{-1}) = \varphi_\tau(\sigma\varphi_\tau(\sigma^{-1})) \]
%has to hold by looking at the first components. This simplifies to $\sigma^{1-k} = \sigma^{k(1-k)}$. Since $\sigma$ has order $p^n$, this implies that
%\[ (k-1)^2 \equiv 0\hspace{-2mm} \pmod{p^n},\mbox{ so in particular }k\equiv 1\hspace{-2mm}\pmod{p}.\]
%Note that $p$ is odd because $q\mid p-1$. The above then yields that the multiplicative order of $k$ mod $p^n$ is a power of $p$, which is a contradiction since $p\neq q$. 
%Hence, the group $C_{p^n}\times C_{q^m}$ is bad when $q\mid p-1$.

\vspace{2mm}

We can construct the associated Hopf--Galois structure explicitly: observe that
\[ \gamma_{(\sigma,1)}(\sigma^i,\tau^j) = (\sigma^i,\tau^{j}),\quad
\gamma_{(1,\tau)}(\sigma^i,\tau^j) = (\sigma^{i\kappa},\tau^j)\]
for all $i,j\in \Z$. This means that if $L/K$ is a $(C_{p^n}\rtimes_\varphi C_{q^m})$-Galois extension, then we obtain the Hopf--Galois structure
\[ H = \left\{
\sum_{i\in \mathbb{Z}/p^n\mathbb{Z}}\sum_{j\in\mathbb{Z}/q^m\mathbb{Z}}\ell_{i,j}(\sigma^i,\tau^j)\in L[C_{p^n}\times C_{q^m}]\colon
\begin{array}{c}
\sigma(\ell_{i,j}) = \ell_{i,j}, \\ \tau(\ell_{i,j})=\ell_{i\kappa,j}\mbox{ for all }i,j \end{array}\right\}\]
of type $C_{p^n}\times C_{q^m}$ for which the Hopf--Galois correspondence is not bijective.
\end{example}

Finally, we consider direct products of cyclic $p$-groups for an odd prime $p$.

\begin{example}\label{ex:Zp} Consider the group $C_{p^n}\times C_{p^m}=(C_{p^n}\times C_{p^m},\cdot)$ with presentation
\[ C_{p^n}\times C_{p^m} = \langle \sigma,\tau \mid \sigma^{p^n}=\tau^{p^m}=1,\ \sigma\tau=\tau\sigma\rangle,\]
where $p$ is a prime and $1\leq m\leq n$. Since $m\leq n$, we can define
\[ \sigma^{i}\tau^{j}\circ \sigma^{r}\tau^{s} = \sigma^{i+r}\tau^{j+s+ir}\]
for all $i,j,r,s\in \Z$, and this yields a brace $(C_{p^{n}}\times C_{p^m},\cdot,\circ)$ by a variation of \cite[Example 6.7]{ST2}. Suppose now that $p$ is odd. Then as mentioned in \cite{ST2}, we have
\[ C_{p^n}\times C_{p^m}\simeq (C_{p^n}\times C_{p^m},\circ)\mbox{ via } \sigma^i\tau^j \mapsto \sigma^i\tau^{j + \frac{i(i-1)}{2}}. \]
Clearly $\langle\sigma\rangle$ is a subgroup of $(C_{p^n}\times C_{p^m},\cdot)$ but not of $(C_{p^n}\times C_{p^m},\circ)$, so not all subgroups of $(C_{p^n}\times C_{p^m},\circ)$ are left ideals of $(C_{p^n}\times C_{p^m},\cdot,\circ)$ by Remark \ref{ex:subgp}. We have shown that the group $C_{p^n}\times C_{p^m}$ is bad when $p$ is odd. 

\vspace{2mm}

We can construct the associated Hopf--Galois structure explicitly: observe that
\[ \gamma_{\sigma}(\sigma^i\tau^j) = \sigma^i\tau^{j+i},\quad
\gamma_{\tau}(\sigma^i\tau^j) = \sigma^{i}\tau^j\]
for all $i,j\in \Z$. This means that if $L/K$ is a $(C_{p^n}\times C_{p^m})$-Galois extension and we identify its Galois group with $(C_{p^n}\times C_{p^m},\circ)$, then we obtain the Hopf--Galois structure
\[ H= \left\{
\sum_{i\in \mathbb{Z}/p^n\mathbb{Z}}\sum_{j\in \mathbb{Z}/p^m\mathbb{Z}}\ell_{i,j}\sigma^i\tau^j\in L[C_{p^n}\times C_{p^m}]\colon
\begin{array}{c}
\sigma(\ell_{i,j}) = \ell_{i,j+i}, \\  \tau(\ell_{i,j})=\ell_{i,j}\mbox{ for all }i,j\end{array}
\right\}\]
of type $C_{p^n}\times C_{p^m}$ for which the Hopf--Galois correspondence is not bijective.
\end{example}

\section{Proof of Theorem \ref{thm:N}}

Let $N = (N,\cdot)$ be any finite group. As already noted in Problem \ref{problem}, it follows from Theorem \ref{thm:corr2} that we are reduced to proving the equivalence of the following.
\begin{enumerate}
\item The group $N$ is good, that is to say, every subgroup of $(N,\circ)$ is a left ideal of $(N,\cdot,\circ)$ for all operations $\circ$ such that $(N,\cdot,\circ)$ is a skew brace. 
\item The group $N$ is isomorphic to $C_2$ or $C_2\times C_2$, or is cyclic of odd order and $q\nmid p-1$ for all prime divisors $p,q$ of $|N|$.
\end{enumerate}
Moreover, as already explained in Section \ref{sec:intro}, by considering the almost trivial skew brace $(N,\cdot,\cdot^{\mbox{\tiny op}})$ , which corresponds to the canonical nonclassical structure, we may assume that $N$ is Hamiltonian or abelian. %Plainly $N$ is good when $|N|\leq 2$, so let us further assume that $|N|\geq 3$. 

\vspace{2mm}

For the implication (1) $\Rightarrow$ (2), observe that:
\begin{itemize}
\item If $N$ is Hamiltonian, then it admits $Q_8$ as a direct factor.
\item If $N$ is elementary $2$-abelian of rank at least $3$, then it admits $C_2\times C_2\times C_2$ as a direct factor.
\item If $N$ is abelian of even order but not elementary $2$-abelian, then it admits $C_{n}$ as a direct factor for some even integer $n\geq 4$.
\item If $N$ is abelian of odd order but not cyclic, then it admits $C_{p^n}\times C_{p^m}$ as a direct factor for some odd  prime $p$ and integers $1\leq m\leq n$.
\item If $N$ is cyclic of odd order and  $q\mid p-1$ for some prime divisors $p,q$ of $|N|$, then it admits $C_{p^n}\times C_{q^m}$ as a direct factor for some integers $m,n\geq 1$.
\end{itemize}
In all cases, we deduce from the examples in Section \ref{sec:ex} that $N$ admits a direct factor that is a bad group. It then follows from Lemma \ref{lem:direct prod} that $N$ is also bad, as desired.

\vspace{2mm}

For the implication (2) $\Rightarrow$ (1), the case $N\simeq C_2$ is obvious because every brace of order $2$ is trivial (or equivalently, only the classical structure arises on a Galois extension of degree $2$). The case $N\simeq C_2\times C_2$ is also easy because the braces of order $4$ have already been classified in \cite[Proposition 2.4]{p3}. Other than the trivial brace (which corresponds to the classical structure), there is only one brace with additive group $C_2\times C_2$ (up to isomorphism), namely the brace
\[ (\Z/2\Z\times \Z/2\Z,+,\circ),\mbox{ where }
\begin{cases}
 (i,j)+(r,s) = (i+r,j+s)\\(i,j)\circ(r,s) = (i+r+js,j+s)
\end{cases}
\]
for all $i,j,r,s\in\Z/2\Z$. Clearly $(\Z/2\Z\times \Z/2\Z,\circ)$ is cyclic generated by $(0,1)$, and its unique non-trivial proper subgroup $\{(0,0),(1,0)\}$ is a left ideal of the brace under consideration. Hence, both of the groups $C_2$ and $C_2\times C_2$ are good.

\vspace{2mm}

Next, suppose that $N$ is cyclic of odd order and let $\circ$ be an operation such that $(N,\cdot,\circ)$ is a brace. We know from \cite[Theorem 1]{Rump} (also see \cite[Remark 1.7]{Tsang}) that necessarily $(N,\circ)$ is a \emph{$C$-group}, meaning that all Sylow subgroups are cyclic. Then by \cite[Lemma 3.5]{Murty}, this implies that
\[ (N,\circ) \simeq C_e\rtimes C_d\mbox{ for some $d,e\in \N$ with }\gcd(d,e)=1.\]
Thus, in the case that $q\nmid p-1$ for all prime divisors $p,q$ of $|N|$, the above must be a direct product. It follows that $(N,\circ)\simeq (N,\cdot)$, and we see from Remark \ref{ex:char subgp} that the subgroups of $(N,\circ)$ are all left ideals of $(N,\cdot,\circ)$. This proves that $N$ is a good group, as claimed.

\vspace{2mm}

This completes the proof of Theorem \ref{thm:N}.

\section*{Acknowledgements}

This research is supported by JSPS KAKENHI Grant Number 24K16891.

We would like to thank the referee for helpful comments.

\end{document}